\documentclass[12pt]{amsart}

\usepackage{amscd,amssymb,amsthm,amsmath,textcomp,supertabular,longtable,enumerate,rotating,mathtools,hyperref,latexsym,amscd,amsbsy,mathrsfs,tikz-cd,pdflscape,xcolor}
\usepackage[matrix,arrow,curve]{xy}
\usepackage[style=alphabetic]{biblatex}
\definecolor{maroon}{rgb}{0,0.,0}
\hypersetup{
    colorlinks=true,
    linkcolor=maroon,  
    citecolor=maroon, 
    urlcolor=magenta 
}

\newtheorem{theorem}[equation]{Theorem}

\newtheorem{lemma}[equation]{Lemma}

\theoremstyle{definition}
\newtheorem{example}[equation]{Example}

\theoremstyle{remark}
\newtheorem{remark}[equation]{Remark}

\makeatletter\@addtoreset{equation}{section} \makeatother

\swapnumbers

\newtheoremstyle{dotless}{}{}{\rm}{}{\sc}{}{ }{}

\theoremstyle{dotless}

\author{Joseph Malbon}

\title{Automorphisms of Fano Threefolds of Rank 2
and Degree 28}

\address{\textnormal{School of Mathematics, The University of Edinburgh,  Edinburgh, UK.}
\newline
\textnormal{\texttt{j.malbon@sms.ed.ac.uk}}}

\renewcommand{\P}{\mathbb{P}}
\renewcommand{\O}{\mathcal{O}}

\newcommand{\Z}{\mathbb{Z}}
\newcommand{\C}{\mathbb{C}}
\newcommand{\Ga}{\mathbb{G}_a}
\newcommand{\Gm}{\mathbb{G}_m}

\newcommand{\PGL}{\mathrm{PGL}_2(\mathbb{C})}
\newcommand{\Aut}{\mathrm{Aut}(Q,C_4)}

\begin{document}

\begin{abstract}
We describe the automorphism groups of smooth Fano threefolds of rank 2 and degree 28 in the cases where they are finite. 
\end{abstract}

\maketitle

\section{Introduction}

A smooth Fano threefold of Picard rank 2 and degree 28 is the blow-up of a smooth quadric threefold $Q\subset\P^4$ in a smooth rational quartic curve $C_4\subset Q$. The family of all such threefolds is denoted \textnumero2.21. Let $\pi\colon X\to Q$ be such a threefold. Then the action of $\text{Aut}(Q,C_4)$ on $Q$ lifts to an action on $X$, so that we may identify it with a subgroup of $\text{Aut}(X)$. \\ 

By a result of Cheltsov-Przyjalkowski-Shramov (\cite[Lemma 9.2]{CPS}), we have that either $\text{Aut}(X)$ is finite, or $\text{Aut}(X)\cong\text{Aut}(Q,C_4)\times\Z_2$, where upto isomorphism $\text{Aut}(Q,C_4)$ is described as follows:

\begin{enumerate}
    \item There is a unique smooth threefold in \textnumero2.21 such that $\text{Aut}(Q,C_4)\cong \PGL$,
    \item There is a one-dimensional family of non-isomorphic smooth threefolds in \textnumero2.21 such that $\text{Aut}(Q,C_4)\cong\Gm\rtimes\Z_2$,
    \item There is a unique smooth threefold in \textnumero2.21 such that $\text{Aut}(Q,C_4)\cong\Ga\rtimes\Z_2$.
    
\end{enumerate}

The goal of this paper is to describe $\text{Aut}(X)$ when it is finite, where $X$ is a smooth threefold in family \textnumero2.21. Our main result is the following:

\begin{theorem}
\label{theorem:main}
Let $X$ be a smooth Fano threefold of rank 2 and degree 28. Then $\mathrm{Aut}(X)\cong\mathrm{Aut}(Q,C_4)\times\Z_2$. Furthermore, if $\mathrm{Aut}(Q,C_4)$ is finite then it is isomorphic to $\Z_2\times\Z_2,\Z_2$ or $0.$
\end{theorem}

We prove this theorem in two parts: Theorem \ref{Aut(Q,C)} and Theorem \ref{bir invol}.

\begin{remark}
The factor of $\Z_2$ appearing in the factorisation $\mathrm{Aut}(X)\cong\mathrm{Aut}(Q,C_4)\times\Z_2$ is generated by an involution $g$, which may be described as follows: 

Let $\mathfrak{d}$ denote the restriction to $Q$ of the linear system of quadric hypersurfaces in $\P^4$ which contain $C_4$, and let $\phi\colon Q\dashrightarrow\P^4$ be the corresponding rational map. The image of $\phi$ is a smooth quadric threefold, and $\phi$ contracts the intersection of the secant variety of $C_4$ with $Q$, $V$, onto a smooth rational curve $C_4'\subset Q'$. The base locus of $\phi$ is equal to $C_4$, so that there is a birational morphism $\pi'\colon X\to Q'$, where $X$ is the blow-up of $Q$ along $C_4$. This morphism contracts the strict transform, $E'$, of $V$ onto the curve $C_4'$. Thus, there is a commutative diagram
$$\begin{tikzcd}
	& X \\
	Q && {Q'}
	\arrow["\pi"', from=1-2, to=2-1]
	\arrow["{\pi'}", from=1-2, to=2-3]
	\arrow["\phi", dashed, from=2-1, to=2-3].
    \end{tikzcd}$$

In \cite{CPS}, it is shown that in cases (1) and (2) of the above classification, there exists a basis of $\mathfrak{d}$ such that $Q'=Q$ and $C_4'=C_4$, so that $\phi$ lifts to an involution $g\in\mathrm{Aut}(X)$. We will show in Theorem \ref{bir invol} that this is always the case. \\
\end{remark}

We can explicitly describe the threefolds appearing in \cite[Lemma 9.2]{CPS}. Let us fix some notation. Observe that after a projective transformation $C_4$ is the image of the Veronese embedding of $\P^1$ in $\P^4$:
\begin{align*}
    \P^1&\to\P^4 \\ 
    [u:v]&\mapsto[u^4:u^3v:u^2v^2:uv^3:v^4].
\end{align*}
The space of global sections of $\mathcal{I}_{C_4}(2)$ is generated by the~following quadratic forms:
\begin{align*}
f_0&= x_3^2 - x_2x_4, \\
f_1&= x_2x_3 - x_1x_4,\\ 
f_2&= x_2^2 - x_0x_4,\\ 
f_3&= x_1x_2 - x_0x_3,\\ 
f_4&= x_1^2 - x_0x_2,\\ 
f_5&= 3x_2^2 - 4x_1x_3 + x_0x_4,
\end{align*}
where $\mathcal{I}_{C_4}$ is the ideal sheaf of $C_4$ in $\P^4$, and $x_0$, $x_1$, $x_2$, $x_3$, $x_4$ are homogeneous coordinates on $\mathbb{P}^4$. Observe that the standard $\PGL$-action on $C_4$ lifts to an action on $\P^4$ such that $C_4$ is invariant. We fix the following subgroups of $\PGL$:
\begin{align*}
&\Z_2,\text{ generated by } \begin{pmatrix}0 & 1 \\ 1 & 0\end{pmatrix},\\
&\Gm,\text{ consisting of matrices } \begin{pmatrix}1 & 0 \\ 0 & t\end{pmatrix} \text{ for every } t\in\Gm,\\
&\Ga,\text{ consisting of matrices } \begin{pmatrix}1 & t \\ 0 & 1\end{pmatrix} \text{ for every } t\in\Ga.
\end{align*}
Now we can describe $\mathrm{Aut}(X)$ for the threefolds listed before:
\begin{example}(\cite[Section 5.9]{book})\textbf{.}
\label{Gm invariant}
     Let $Q$ be the quadric given by the equation 
$$(1-4s^2)f_2+f_5=0,$$
for some $s\in\C\setminus\{-1,0,1\}$. Then $Q$ is $\Gm$-invariant and $\Z_2$-invariant, and conversely any smooth quadric admitting a faithful $\Gm$-action is isomorphic, via an element of $\PGL$, to a quadric given by an equation of this form. Moreover, we have the following:
$$\mathrm{Aut}(Q,C_4)\cong
\begin{cases}
    \Gm\rtimes\Z_2, & s\neq\pm\frac{1}{2}, \\ 
    \PGL, & s=\pm\frac{1}{2}.
\end{cases}$$

The involution $g$ described before is given by:
$$\tau\colon[x_0:x_1:x_2:x_3:x_4]\mapsto[f_4:sf_3:s^2f_2:sf_1:f_0].$$
See \cite[Remark 5.52]{book} for an explanation of why $\tau\circ\tau\colon Q\dashrightarrow Q$ is the identity map on $Q\setminus C_4$.

\end{example}

\begin{example}
\label{Ga invariant}
    Suppose that the quadric $Q$ is given by the equation
    $$f_0+f_5=0.$$
    Then $Q$ is $\Ga$-invariant and $\Z_2$-invariant, and $\mathrm{Aut}(Q,C_4)\cong\Ga\rtimes\Z_2$. We will prove in case (2) of Theorem \ref{bir invol} that the blow-up of $Q$ in $C_4$ admits an action of the involution $g$. \\ 
\end{example}

\begin{remark} Recall that for a finite subgroup $G\subset\text{Aut}(Y)$, a variety $Y$ is called \textit{$G$-Fano} if it has terminal singularities, $-K_Y$ is ample and $\text{Cl}(Y)^G$ is rank 1. It is proven in \cite{Prokhorov} that the Hilbert scheme of conics on a smooth threefold $X$ from the family \textnumero2.21 is isomorphic to $\P^1\times\P^1$, with the degenerate conics being parameterised by a smooth curve $C\subset\P^1\times\P^1$ of bidegree $(2,2)$. If $X$ is $G$-Fano for the group $G=\langle g\rangle\cong\Z_2$, then this curve must be invariant upon swapping the two factors of $\P^1$.

An informal conjecture of Y. Prokhorov is that the invariance of this curve is a sufficient condition for $X$ to be $G$-Fano. It is proven in \cite{Druel} that every smooth curve in $\P^1\times\P^1$ of bidegree $(2,2)$ is invariant. As a corollary to Theorem \ref{theorem:main}, we have that every smooth threefold $X$ in the family \textnumero2.21 is $G$-Fano, so that Prokhorov's informal conjecture is true. For a detailed discussion of $G$-Fano threefolds, see \cite{Arman}. \\
\end{remark}

\begin{remark}
    Smooth threefolds in the family \textnumero2.21 are parametrised by $\P^5\setminus\Delta$, where $\Delta\subset\P^5$ is the discriminant locus of singular quadrics. The group $\PGL$ acts on this space, and it follows from Theorem \ref{theorem:main} that any two threefolds in \textnumero2.21 are isomorphic if and only if their corresponding points in the parameter space $\P^5\setminus\Delta$ lie in the same $\PGL$-orbit. Moreover, the moduli space of smooth GIT-polystable threefolds in \textnumero2.21 is given by the GIT quotient
$$(\P^5\setminus\Delta)//\PGL.$$
\end{remark}

\textbf{Acknowledgements.}
I am grateful to I. Cheltsov for introducing me to this topic, and for his attentive guidance, support and insight with regards to the creation of this document. I would also like to thank I. Dolgachev for his useful observation in the proof of case (1) in Theorem \ref{bir invol}. \\ 

\section{Computation of $\text{Aut}(Q,C_4)$}
The first half of proving Theorem \ref{theorem:main} is the computation of $\Aut$, which we will do in this section. The result we will prove is: 

\begin{theorem}
\label{Aut(Q,C)}
    Let $Q$ be a smooth quadric threefold containing the quartic curve $C_4$. If $\Aut$ is finite, then it is isomorphic to either $\Z_2\times\Z_2$, $\Z_2$ or $0$.
\end{theorem}
The following lemma will be useful:

\begin{lemma}
\label{Zn}
    Let $Q\subset\P^4$ be any quadric hypersurface containing the curve $C_4$. Suppose that $\Aut$ is finite, and contains an element of finite order $n>2$. Then $Q$ is singular.
\end{lemma}

\begin{proof}
Since $\Aut\subseteq\mathrm{Aut}(\P^4,C_4)\cong\PGL$, we may identify $\Aut$ with a subgroup of $\PGL$. Moreover, by considering the action of $\PGL$ on the parameter space $\P^5$, we identify $\Aut$ with the stabiliser of the point of $\P^5$ corresponding to $Q$. Fix a finite cyclic subgroup $G\subset\text{PGL}_2(\C)$ of order $n$, and let $g_1\in G$ be a generator. Then $g_1$ fixes precisely two distinct points of $\P^1$, which upto projective transformation are $[0:1]$ and $[1:0]$. Thus
$$g_1=\begin{pmatrix} 1 & 0 \\ 0 & \zeta\end{pmatrix},$$
for some primitive $n^\text{th}$ root of unity $\zeta$. Then $g_1$ acts on $\P^5$ by:
$$[a_0:a_1:a_2:a_3:a_4:a_5]\mapsto[\zeta^6a_0:\zeta^5a_1:\zeta^4a_2:\zeta^3a_3:\zeta^2a_4:\zeta^4a_5],$$

and we can read off the points of $\P^5$ whose stabiliser contains $G$:

\begin{itemize}
    \item $n=2$: $(\P^5)^G = \big\{[a_0:0:a_2:0:a_4:a_5],[0:a_1:0:a_3:0:0]\big\}$, 
    \item $n=3$: $(\P^5)^G = \big\{[a_0:0:0:a_3:0:0], [0:a_1:0:0:a_4:0]\big\}$,
    \item $n=4$: $(\P^5)^G = \big\{[a_0:0:0:0:a_4:0]\big\},$
    \item $n>4$: $(\P^5)^G=\varnothing$,
\end{itemize}
where the numbers $a_0,a_1,a_2,a_3,a_4,a_5$ are all arbitrary complex numbers. 
One checks that for $n>2$, the corresponding threefolds which have finite $\Aut$ are all singular.    
\end{proof}

Now let us recall the following classification theorem for quadric threefolds which contain the curve $C_4$:

\begin{theorem}[\cite{K-STAB}]
\label{theorem:classification}
Let $Q\subset\P^4$ be a smooth quadric containing $C_4$. Then there exists an automorphism $\phi\in\PGL$ such that $\phi(Q)$ is given by one of the following equations:
\begin{enumerate}[$(1)$]
\item $\mu(f_0+f_4)+\lambda f_2+f_5=0$, for some $\lambda\in\C\setminus\{1,-3\}$ and $\mu\in\C\setminus\{2,-2\}$ such that $\mu^2\neq-\lambda^2-2\lambda+3$,
\item $f_0+\lambda f_2+ f_5=0$, for some $\lambda\in\mathbb{C}\setminus\{1,-3\}$,
\item $f_1+f_5=0$.
\end{enumerate}
\end{theorem}

Let us find $\text{Aut}(Q,C_4)$ in each of these cases. 
\begin{proof}[Proof of Theorem \ref{Aut(Q,C)}.] We may assume that $\mu\neq0$ in case (1), and $\lambda\neq0$ in case (2), as otherwise the threefolds are isomorphic to those which are described in Example \ref{Gm invariant} and Example \ref{Ga invariant}. Then $\Aut$ is finite, and since $\Aut$ is isomorphic to a subgroup of $\PGL$, it must be isomorphic to one of the following groups:
$$0,\Z_n,\Z_2\times\Z_2,D_{2n},\mathfrak{A}_4,\mathfrak{S}_4,\mathfrak{A}_5
,$$ where $\mathfrak{S}_n$ (resp. $\mathfrak{A}_n$) is the symmetric (resp. alternating) group on $n$ letters. Then by Lemma \ref{Zn} the only possibilities are that $\Aut$ is isomorphic to $\Z_2\times\Z_2, \Z_2$ or $0$.\\

Suppose $Q$ is in case (1). Then $Q$ admits an action of $\Z_2\times\Z_2$, generated by $g_1,g_2\in\PGL$, which are given by:
$$g_1=\begin{pmatrix}1 & 0 \\ 0 & -1\end{pmatrix},g_2=\begin{pmatrix}0 & 1 \\ 1 & 0 \end{pmatrix}.$$

Hence, $\Aut\cong\Z_2\times\Z_2$. \\ 

Suppose that $Q$ is in case (2). Then $Q$ admits an action of the group $\Z_2$, generated by the element $g_1$. Suppose that $\Aut\cong\Z_2\times\Z_2$, and let $g\in\Aut$ be a non-trivial element distinct from $g_1$. Considering the standard action of $\PGL$ on $\P^1$, observe that $g_1$ fixes the points $[0:1]$ and $[1:0]$, and since $gg_1=g_1g$, we see that $g$ must swap these points. Since $g$ has order 2, it must be equal to either $g_2$ or $g_1g_2$. The threefold $Q$ is not invariant under either of these. \\ 

Finally suppose that $Q$ is in case (3), and  suppose that $\Aut$ is non-trivial. Then it contains an element, $g$, of order 2. Since $g$ fixes two distinct points of $\P^1$, it must be equal to $g_2$, $g_1g_2$, or be given by a matrix of the form 
$$\begin{pmatrix}
    1 & a \\ 
    b & -1
\end{pmatrix},\text{ for some $a,b\in\C$ such that $ab\neq-1$}.$$
One checks that $g_2$ nor $g_1g_2$ leave $Q$ invariant, and if $g$ is given by a matrix of the above form then $g(Q)$ is given by the equation:
$$4bf_0+2(1-3bc)f_1-3c(1-bc)f_2+2c^2(3-bc)f_3-4c^3f_4+(bc^2-2b^2c^2-4bc-c-2)f_5=0$$
Clearly $g(Q)\neq Q$, so that $\Aut$ has to be trivial.
\end{proof}

\section{Existence of a birational involution of $Q$}
The second half of proving Theorem \ref{theorem:main} is the assertion that $\mathrm{Aut}(X)\cong\mathrm{Aut}(Q,C_4)\times\Z_2$, which we will do in this section. The result is:
\begin{theorem}
\label{bir invol}
Let $X$ be a smooth Fano threefold in family \textnumero2.21. Then there exists an involution $g\in\mathrm{Aut}(X)$ such that $\mathrm{Aut}(X)\cong \mathrm{Aut}(Q,C_4)\times\langle g\rangle$.
\end{theorem}

\begin{proof}
We proceed case-by-case, according to the classification in Theorem \ref{theorem:classification}. 

\subsection*{Case (1): $Q$ is given by $\mu(f_0+f_4)+\lambda f_2+f_5=0$}
Observe that the linear system of quadrics which contain $C_4$ is 5-dimensional, so it is more natural to express members of family \textnumero2.21 in terms of fourfolds. Let us show how to do this. Fix the Veronese surface $S_4\subset\P^5$ given by the embedding:
    \begin{align*}
    \upsilon\colon\P^2&\to \P^5 \\ 
    [x:y:z]&\mapsto[x^2:xy:y^2:yz:z^2:xz].
    \end{align*}

The space of global sections of $\mathcal{I}_{S_4}(2)$ is generated by the quadratic forms:
    \begin{align*}
g_0 &= x_3^2 - x_2 x_4,\\
g_1 &= x_3 x_5 - x_1 x_4,\\ 
g_2 &= x_5^2 - x_0 x_4, \\ 
g_3 &= x_1 x_5-x_0 x_3,   \\ 
g_4 &= x_1^2 - x_0 x_2, \\ 
g_5 &= x_1 x_3 - x_2 x_5,
\end{align*}
where $x_0,x_1,x_2,x_3,x_4,x_5$ are homogeneous coordinates on $\P^5$. \\ 

Consider the following rational map:
    \begin{align*}
        \phi\colon\P^5&\dashrightarrow\P^5 \\ 
        [x_0:x_1:x_2:x_3:x_4:x_5]&\mapsto [g_0:g_1:g_2:g_3:g_4:g_5].
    \end{align*}
    
I claim that $\phi$ is a birational involution. The following observation is due to I. Dolgachev: we can identify $\P^5$ with the space of symmetric $3\times3$ matrices, upto scaling. Then under this identification, the rational map above is:
    \begin{align*}
	\phi\colon\mathbb{P}^5 &\dashrightarrow \mathbb{P}^5 \\
	\begin{pmatrix} x_0 & x_1 & x_5 \\ x_1 & x_2 & x_3 \\ x_5 & x_3 & x_4\end{pmatrix} &\mapsto \begin{pmatrix} x_3^2 - x_2 x_4 & x_3 x_5 - x_1 x_4 & x_1 x_3 - x_2 x_5 \\ x_3 x_5 - x_1 x_4 & x_5^2 - x_0 x_4 & x_1 x_5-x_0 x_3 \\ x_1 x_3 - x_2 x_5 & x_1 x_5-x_0 x_3 & x_1^2 - x_0 x_2\end{pmatrix}
\end{align*}
But this is the same map as taking a matrix $M$ to its adjoint $\mathrm{adj}(M)$. Thus it follows from the relation $\mathrm{adj}(\mathrm{adj}(A))=\mathrm{det}(A)^{n-2}A$ for any $n\times n$ matrix $A$ that $\phi$ is a birational involution. \\ 

Let $\sigma\colon\widetilde{\P}^5\to\P^5$ be the blow-up of $\P^5$ in $S_4$, and let $E$ be the exceptional divisor. Observe that for general divisors $\widetilde{H}\in|\sigma^*\O_{\P^5}(1)|$ and $\widetilde{Q}\in|\sigma^*\O_{\P^5}(2)-E|$, we have that $\widetilde{H}\cap \widetilde{Q}$ is a smooth element of \textnumero2.21. 

Since $\phi$ has base locus equal to $S_4$, it lifts to a biregular involution $g\in\mathrm{Aut}(\widetilde{\P}^5)$ which swaps the linear systems $|\widetilde{H}|$ and $|\widetilde{Q}|$. Thus, the intersection $\widetilde{H}\cap g(\widetilde{H})$ is $\langle g\rangle$-invariant, for any $\widetilde{H}$. We will now show that every smooth element $X$ of \textnumero2.21 which is in case (1) of Theorem \ref{theorem:classification} is isomorphic to a subvariety of $\widetilde{\P}^5$ of the form $\widetilde{H}\cap g(\widetilde{H})$, for some hyperplane $H\subset\P^5$, and therefore possesses an involution not coming from $\Aut$. \\ 

So fix such a threefold $X$. Then the quadric $Q$ is given by the equation $$\mu(f_0+f_4)+\lambda f_2+f_5=0,$$
for some $\lambda\in\C\setminus\{1,-3\}$ and $\mu\in\C\setminus\{2,-2\}$ such that $\mu^2\neq-\lambda^2-2\lambda+3$. Let us choose roots $a,b$ of the equations 
\begin{align*}
    (\mu+2)x^4+2\lambda-2 &=0, \\ 
    (\mu + 2)x^4+\mu-2 &=0,
\end{align*}
respectively, so that the equation of $Q$ becomes:
\begin{equation}
\label{eq}
    \frac{2 - 2b^4}{1 + b^4}(f_0+f_4)+ \frac{1-2a^4 + b^4}{1 + b^4}f_2 + f_5=0.
\end{equation}
Now consider the following hypersurfaces in $\P^5$:
\begin{align*}
	H&=\{x_0=a^2x_2+b^2x_4\},\\
    Q_2&=\{g_0=a^2g_2+b^2g_4\}.
\end{align*}
We have that the intersections $H\cap S_4$ and $Q_2\cap S_4$ are smooth, so that the intersection of their strict transforms, $\widetilde{H}\cap\widetilde{Q}_2\subset\widetilde{\P}^5$, is a smooth member of \textnumero2.21. Moreover, $Q_2=\phi(H)$, so that $\widetilde{H}\cap \widetilde{Q}_2$ is $\langle g\rangle$-invariant. Consider the projective transformation $\psi\colon\P^5\to\P^5$ given by the matrix\footnote{The matrix defining $\psi$ comes from the embedding $\mathrm{PGL}_3(\C)\hookrightarrow\mathrm{PGL}_6(\C)$, which is given by the projectivisation of the symmetric square, $\P(\mathrm{Sym}^2(\C^3))\cong\P^5$, of the standard $\mathrm{GL}_3(\C)$ action. The Veronese surface $S_4$ is invariant under this action, so that $\psi(Q_2)$ contains $S_4.$}
$$
\begin{pmatrix}
1 & 0 & 0 & 0 & b^2 & -2b \\
0 & -a & 0 & ab & 0 & 0 \\
0 & 0 & a^2 & 0 & 0 & 0 \\
0 & -a & 0 & -ab & 0 & 0 \\
1 & 0 & 0 & 0 & b^2 & 2b \\
1 & 0 & 0 & 0 & -b^2 & 0 \\
\end{pmatrix}.
$$
Then the intersection of the fourfolds
\begin{align*}
    \psi(H)&=\{x_2-x_5=0\}\\
    \psi(Q_2)&=\{(1-b^4)(g_0+g_4)-a^4g_2+2(1+b^4)g_5=0\}
\end{align*}
is given as a subvariety of $\P^4$ by Equation \ref{eq}, which defines $X$. Thus $X\cong \widetilde{H}\cap\widetilde{Q}_2$. \\

It remains to show that the birational involution $g$ commutes with the action of $\text{Aut}(Q,C_4)$. \\ 

Consider the subgroup $G\subset\mathrm{PGL}_6(\C)$ generated by the commuting involutions
    \begin{align*}
        \alpha\colon[x_0:x_1:x_2:x_3:x_4:x_5]&\mapsto[x_0:x_1:x_2:-x_3:x_4:-x_5]\\
        \beta\colon[x_0:x_1:x_2:x_3:x_4:x_5]&\mapsto[x_0:-x_1:x_2:-x_3:x_4:x_5].
    \end{align*}
    Then $\alpha$ and $\beta$ commute with the birational involution described previously:
    \begin{align*}
        \phi\colon[x_0:x_1:x_2:x_3:x_4:x_5]\mapsto [x_3^2 - x_2 x_4&: x_3 x_5 - x_1 x_4:x_5^2 - x_0 x_4: \\ 
        &:x_1 x_5-x_0 x_3: x_1^2 - x_0 x_2:x_1 x_3 - x_2 x_5],
    \end{align*}

    The hypersurfaces $H$ and $Q_2$ are $G$-invariant. Moreover $S_4$ is $G$-invariant, so that $G$ is isomorphic to a subgroup of $\Aut$. Thus since $\mathrm{Aut}(Q,C_4)\cong\Z_2\times\Z_2$ by Theorem \ref{Aut(Q,C)}, we have that $\mathrm{Aut}(Q,C_4)\cong G$. So we see that $\Aut$ commutes with the involution $g$.\\

For the remaining cases, we will compute bases for the linear system $\mathfrak{d}$ of quadrics sections of $Q\subset\P^4$ containing the curve $C_4$ such that the corresponding rational map is an involution, and commutes with the action of $\Aut$. \\

\subsection*{Case (2): $Q$ is given by $f_0+\lambda f_2+f_5=0$}
    Let us make the substitution $\lambda=1-4s^2$, for some $s\in\mathbb{C}\setminus\{-1,0,1\}$. Consider the rational map:
\begin{align*}
    \iota\colon Q&\dashrightarrow\P^4 \\ 
    [x_0:x_1:x_2:x_3:x_4]&\mapsto[f_4+\frac{s^2}{2}f_2-\frac{1}{16}f_0: \frac{s}{4}f_1+sf_3: s^2f_2: sf_1: f_0].
\end{align*} 
    Observe that it has base locus equal to $C_4$, so indeed corresponds to the linear system $\mathfrak{d}$. To see that the map $\iota$ is a birational involution, consider the following rational parametrisation of $Q$, 
    \begin{align*}
        p\colon \P^3&\dashrightarrow Q \\ 
        [x_0:x_2:x_3:x_4]&\mapsto[x_0x_3 : s^2x_0x_4 - s^2x_2^2 + x_2^2 - \frac{1}{4}x_2x_4 + \frac{1}{4}x_3^2 : x_2x_3 : x_3^2 : x_3x_4].
    \end{align*}

   This is a rational inverse to the projection $Q\dashrightarrow\P^3$ from the point $[0:1:0:0:0]$. Moreover, it is an isomorphism between the open subsets $\P^3\setminus\Pi$ and $Q\setminus V$, where $\Pi\subset\P^3$ is the plane given by $x_3=0$, and $V\subset Q$ is the singular quadric surface given by the intersection of $Q$ with the plane $x_3=0$, this latter variety being the closure of the union of lines through $[0:1:0:0:0]$. Let $Z$ be the curve $p^{-1}(C_4)$, which is a quartic rational curve in $\P^3$. \\ 
    
Then $\iota(p(\P^3\setminus (\Pi\cup Z))$ lies in $Q$, and since $p(\P^3\setminus (\Pi\cup Z))=Q\setminus(V\cup C_4)$ is dense in $Q\setminus C_4$, it follows that $\iota$ is a rational self-map of $Q$. To see that $\iota$ is an involution on $Q\setminus C_4$, observe that $\iota\circ\iota\circ p $ is equal to the map 
    \begin{align*}
        \P^3&\dashrightarrow Q \\ 
        [x_0:x_1:x_2:x_3]&\mapsto [x_0: \frac{(-4s^2 + 4)x_2^2 - x_2x_4 + 4s^2x_0x_4 + x_3^2}{4x_3}: x_2: x_3: x_4],
    \end{align*}
    which is equal to the identity morphism on $\P^3\setminus \Pi$. Thus $\iota\circ\iota$ is equal to the identity morphism on $Q\setminus(V\cup C_4)$, so that it is equal to the identity morphism on $Q\setminus C_4.$ \\

Let us prove that $\iota$ commutes with the action of $\Aut$. If $s\neq\pm\frac{1}{2}$ then by Theorem \ref{Aut(Q,C)}, $\Aut=\langle g_1\rangle$, where $g_1$ is the linear transformation
$$[x_0:x_1:x_2:x_3:x_4]\mapsto[x_0:-x_1:x_2:-x_3:x_4].$$
Then it is plain that $\iota$ commutes with $g_1$. If $s=\pm\frac{1}{2}$, then $Q$ is the quadric described in Example \ref{Ga invariant}, and $\Aut$ contains subgroup isomorphic to $\Ga$ consisting of automorphisms of the form
\begin{align*}
[x_0:x_1:x_2:x_3:x_4]\mapsto[x_0+4tx_1+6t^2x_2+4t^3x_3+t^4x_4&:x_1+3tx_2+3t^2x_3+x_4: \\ 
&:x_2+2tx_3+t^2x_4:x_3+tx_4:x_4],
\end{align*}
for every $t\in\C$. One can easily see that $\iota$ commutes with each of these automorphisms.

\subsection*{Case (3): $Q$ is given by $f_1+f_5=0$}
    Let $\sigma$ be the rational map 
\begin{align*}
        \sigma\colon Q&\dashrightarrow\P^4 \\ 
        [x_0:x_1:x_2:x_3:x_4]&\mapsto[\frac{1}{64}f_0 - \frac{1}{4}f_1 + \frac{3}{2}f_2 + 16f_3 + 64f_4 :\\
        &:-\frac{1}{8}f_0 + \frac{3}{2}f_1 + 2f_2 + 32f_3 :f_0 - 8f_1 + 16f_2:-8f_0 + 32f_1:64f_0].
\end{align*}
     By replacing the map $\iota$ with $\sigma$, the map $p$ by the rational parametrisation
    \begin{align*}
        \P^3&\dashrightarrow Q \\ 
        [x_1:x_2:x_3:x_4]&\mapsto[\frac{4x_1x_3 + x_1x_4 - 3x_2^2 - x_2x_3}{x_4}:x_1:x_2:x_3:x_4],
    \end{align*}
    and the varieties $\Pi$ and $V$ with $\{x_4=0\}$,
    it follows verbatim from the proof of case (2) that $\sigma$ is a birational involution of $Q$ with base locus equal to $C_4$. By Theorem \ref{Aut(Q,C)}, $\Aut$ is trivial in this case, so there is nothing left to prove. 
\end{proof}

\printbibliography

\end{document}